\documentclass[12pt]{amsart}
\usepackage{verbatim,amsmath,latexsym,amssymb,amsbsy, amscd,graphicx, amsrefs, hyperref}

\hypersetup{colorlinks,  citecolor=blue, linkcolor=blue, urlcolor=blue}

\newtheorem{Theorem}{Theorem}[section]
\newtheorem{Lemma}[Theorem]{Lemma}
\newtheorem{Corollary}[Theorem]{Corollary}
\newtheorem{Proposition}[Theorem]{Proposition}

\newtheorem{Notation and Discussion}[Theorem]{Notation and Discussion}
\newtheorem{Assumptions and Discussion}[Theorem]{Assumptions and Discussion}

\newtheorem{Question}[Theorem]{Question}

\theoremstyle{definition}
\newtheorem{remark}[Theorem]{Remark}

\def\m{{\mathfrak m}}
\def\det{{\rm det}}
\def\adj{{\rm adj}}

\title{
On the Equality of Ordinary and Symbolic Powers of Ideals
}
\author[A. Hosry, Y. Kim and J. Validashti]{Aline Hosry, Youngsu Kim and Javid Validashti}

\address{Department of Mathematics and Statistics, Notre Dame University-Louaize, P.O. Box:~72, Zouk Mikael, Zouk Mosbeh, Lebanon}
\email{ahosry@ndu.edu.lb}
\address{Department of Mathematics, Purdue University, West Lafayette, IN 47907}
\email{kim455@purdue.edu}
\address{Department of Mathematics, University of Illinois at Urbana-Champaign, IL 61801}
\email{jvalidas@illinois.edu}

\begin{document}

\maketitle
\begin{abstract}
We consider the following question concerning the equality of ordinary and symbolic powers of ideals. In a regular local ring, if the ordinary and symbolic powers of a one-dimensional prime ideal are the same up to its height, then  are they the same for all powers? We provide supporting evidence of a positive answer for classes of prime ideals defining  monomial curves or rings of low multiplicities.
\end{abstract}

\section{Introduction}
Let $R$ be a  Noetherian local  ring of dimension $d$ and $P$ a  prime  ideal of $R$. For a positive integer $n$, the $n$-th symbolic power of $P$, denoted by $P^{(n)}$, is defined as 
$$P^{(n)}:=P^nR_P \cap R=\{ x\in R \  | \  \exists \, s \in R\setminus P, \  sx\in P^n \}.$$ 
Symbolic powers of ideals are central objects in commutative algebra and algebraic geometry for their tight connection to primary decomposition of ideals and the order of vanishing of polynomials. 
One readily sees from the definition that $P^n \subseteq P^{(n)}$ for all $n$, but  they are not equal in general.
Therefore, one would like to compare the ordinary and symbolic powers and provide criteria for equality.
This problem has long been a subject of interest, see for instance \cite{BoH,ELS, Ho, HH, HH1, HuH, HKV, M}.
In this paper, we are interested in criteria for the equality. In particular, we would like to know if $P^n = P^{(n)}$ for all $n$ up to some value, then they are equal for all $n$. The following question was posed by Huneke in this regard.

\begin{Question}\label{Q1}
Let $R$ be a regular local ring  of dimension $d$ and $P$ a prime ideal of height $d-1$. If $P^{n}  = P^{(n)}$ for all $n \leq {\rm ht}\, P$, then is $P^{n} = P^{(n)}$ for all $n$?
\end{Question}

An affirmative answer to Question \ref{Q1} is equivalent to $P$ being generated by a regular sequence \cite{CN}. Furthermore, it is equivalent  to showing that if $P^{n}  = P^{(n)}$ for all $n \leq d-1$, then the analytic spread of $P$ is  $d-1$. 
This is not known even when $P$ is the defining  ideal of a monomial curve in $\mathbb{A}^4_{\textsf{k}}$.
 Huneke answered  Question \ref{Q1} positively in dimension $3$, and in dimension $4$ if $R/P$ is Gorenstein  \cite[Corollaries 2.5, 2.6]{Hu}. One would like to remove the Gorenstein assumption. There are supporting examples showing that the Gorenstein property of $R/P$ might follow from $P^2=P^{(2)}$.
In fact, this is very close to a conjecture by Vasconcelos which states that if $P$ is syzygetic and $R/P$ and $P/P^2$ are Cohen-Macaulay, then $R/P$ is Gorenstein \cite{Vas}. Note that if $P$ has height $d-1$, then $R/P$ is Cohen-Macaulay, and $P/P^2$ being Cohen-Macaulay is equivalent to $P^{2} = P^{(2)}$. Therefore, one is tempted to ask the following question.

\begin{Question}\label{con1}
Let $R$ be a regular local ring  of dimension $d$ and $P$ a prime ideal of height $d-1$. If $P^{2} = P^{(2)}$, then is $R/P$ Gorenstein?
\end{Question}

By Huneke's result  \cite[Corollary 2.6]{Hu}, if  Question \ref{con1} has an affirmative answer, then so does Question \ref{Q1} when dimension  of $R$ is $4$.
The converse of  Question \ref{con1} is true in dimension $4$ by Herzog \cite{He}. Also, Question \ref{con1} has been answered positively for some classes of algebras \cite{MX}, but it is not true in general (see for instance \cite[Example 6.1]{MX}).
In this paper, we consider the case where $P$ is the defining ideal of a monomial curve $\textsf{k}[[t^{a_1}, \ldots, t^{a_d}]]$ and we give an affirmative answer to Questions \ref{Q1} and  \ref{con1} when $d=4$ and $\{ a_i\}$ forms an arithmetic sequence. 
In higher dimensions, if $\{a_i\}$ contains an arithmetic subsequence of length $5$ in which the terms are not necessarily consecutive, we observe that  $P^{2} \not = P^{(2)}$, hence we have a positive answer to Questions \ref{Q1} and  \ref{con1}. We extend these results to certain modifications of arithmetic subsequences. We also consider one-dimensional prime ideals $P$ of a regular local ring $R$ in general and we show that if $R/P$ has low multiplicity, then Question \ref{Q1} has a positive answer. We note that if we drop the height $d-1$ assumption on $P$, then this question does not have a positive answer in general, due to a counterexample by Guardo, Harbourne and Van Tuyl \cite{GHV}.

\vskip .3in
\section{Monomial Curves}

Let $a_1, \ldots, a_d$ be an increasing sequence of positive integers with ${{\rm gcd}(a_1, \ldots, a_d)=1}$. Assume that  $a_i$'s generate a numerical semigroup   non-redundantly. Consider the   monomial curve 
$$
A=\textsf{k}[[t^{a_1}, \ldots, t^{a_d}]] \subset \textsf{k}[[t]]
$$ 
over a field $\textsf{k}$, with maximal ideal  $\m_A := (t^{a_1}, \ldots, t^{a_d})A$.  Let $R=\textsf{k}[[x_1, \ldots, x_d]]$ be a formal power series ring with maximal ideal $\m=(x_1, \ldots, x_d)R,$ and let $P$ be the kernel of the homomorphism
$$
 \textsf{k}[[x_1, \ldots, x_d]]  \longrightarrow \textsf{k}[[t^{a_1}, \ldots, t^{a_d}]]
$$
obtained by mapping $x_i$ to $t^{a_i}$ for all $i$. Therefore, $A$ is isomorphic to $R/P$.
Note that $ P \subset \m^2$ because of the non-redundancy assumption on  ${a_i}$'s. We state the following  well-known properties about monomial curves.

\begin{Lemma}\label{folklore}
In the above setting,
\begin{enumerate}
 \item The ideal $t^{a_1}A$ is a minimal reduction of $\m_A$.
 \item The Hilbert-Samuel multiplicity $e(\m_A,A)$ of $A$ is  $a_1$.
 \item The multiplicity $e(\m_A,A)$ is at least $d$, i.e., $a_1\geq d$.
\end{enumerate}

\end{Lemma}
For the third property in above, we may assume that $\textsf{k}$ is an infinite field. Then by \cite[Fact (1)]{Abh}, we have $e(\m_A,A) \ge \lambda(\m_A/\m_A^2)$, where $ \lambda(-)$ denotes the length, and observe that $\lambda(\m_A/\m_A^2) = d$, by the non-redundancy condition on the $a_i$'s.  Note that the third property  also follows from Theorem \ref{lowd}. We begin with the following result that describes  the generators of $P$ when $d=4$ and the set of exponents $\{a_i\}$ forms an arithmetic sequence.

\begin{Proposition}\label{gens}
Let $A$ be the monomial curve $\textsf{k}[[t^a, t^{a+r}, t^{a+2r}, t^{a+3r}]]$, where $a$ and $r$ are positive integers that are relatively prime. Regard $A$ as $R/P$, where  $R=\textsf{k}[[x,y,z,w]]$ and $P$ is the defining ideal of $A$. Then $P$ is minimally generated by 
\begin{align*}
&z^2-yw,\; yz-xw,\; y^2-xz,\; x^{k+r}-w^k, &\hbox{if} \quad a = 3k \ \ \ \ \  \\
&z^2-yw,\; yz-xw,\; y^2-xz,\; x^{k+r}z-w^{k+1},\; x^{k+r}y-zw^k,\; x^{k+r+1}-yw^k,  &\hbox{if}  \quad a = 3k + 1\\
&z^2-yw,\; yz-xw,\; y^2-xz,\; x^{k+r+1}-zw^k,\; x^{k+r}y-w^{k+1},  &\hbox{if} \quad a = 3k + 2
\end{align*}
where $k$ is a positive integer.
\end{Proposition}

\begin{proof} Since the numerical semigroup is non-redundantly generated, $a$ is greater than or equal to $4$ by Lemma ~\ref{folklore}. Thus $k \ge 2$ if $a = 3k$ and $k \ge 1$ if $a = 3k+1$ or $3k+2$.  In each case, let $I$ be the ideal  generated by the above-listed elements and $\m$ be the maximal ideal $ (x,y,z,w)$ of $R$. One can  directly check that  $ I \subset P$. For all cases, we will use the following method to show  $I=P$.  First, we show that $(P,x) = (I,x)$. Then it follows that $P = I + x(P:x)$. But $(P:x) = P$, since $x$ is not in $P$. Thus $P = I + xP$, which implies $P = I$, by Nakayama's Lemma.  To show $(P,x) = (I,x)$, let $\tilde{I} = (I,x)$. The short exact sequence
$$0 \longrightarrow R/(\tilde{I} :y) \xrightarrow{ \ \ \cdot y \ \, } R/\tilde{I} \longrightarrow R/(\tilde{I},y) \longrightarrow 0 $$
yields the length equation $ \lambda_R(R/\tilde{I}) =  \lambda_R(R/(\tilde{I}:y)) +  \lambda_R(R/(\tilde{I}, y))$. Since $R/P$ is Cohen-Macaulay and the image of the ideal $(x)$ in $R/P$  is a minimal reduction of $\m/P$ by Lemma ~\ref{folklore}, we have
$$
a=e(\m,R/P) =    \lambda_R(R/(P,x)) \le  \lambda_R(R/\tilde{I}).
$$ 
Thus, it is enough to show 
$$ 
 \lambda_R(R/(\tilde{I}:y)) +  \lambda_R(R/(\tilde{I}, y)) \leq a.
$$
If $ a = 3k$, then $\tilde{I} = (x, z^2-yw, yz,y^2,w^k)$. Therefore, $(\tilde{I}, y) = (x,y, z^2,w^k)$ and  the ideal $\tilde{I}:y$ contains the ideal $(x,y,z,w^k)$. Thus,
$$  \lambda_R(R/(\tilde{I}:y)) +  \lambda_R(R/(\tilde{I}, y)) \leq  \lambda_R(R/(x,y,z,w^k))  +  \lambda_R(R/(x,y,z^2,w^k))  \leq  k+ 2k = a.$$
If $ a = 3k + 1$, then $\tilde{I} = (x, z^2-yw, yz, y^2, w^{k+1},zw^k,yw^k)$. Hence, 
$ (x,y,z,w^k) \subset \tilde{I}:y $ and $(\tilde{I}, y) = (x, y, z^2,  zw^k,w^{k+1})$. Note that $ \lambda_R(R/(x, y, z^2, zw^k, w^{k+1}))  = 2k+1$ 
and $ \lambda_R(R/(x,y,z,w^k)) = k$. Thus, 
$$ \lambda_R(R/(\tilde{I}:y)) +  \lambda_R(R/(\tilde{I}, y)) \leq k+(2k+1) = a.$$
If $a = 3k + 2$, then $\tilde{I} = (x, z^2-yw,yz, y^2,zw^k, w^{k+1})$. Therefore, $ (x,y,z,w^{k+1})\subset \tilde{I}:y $ and $(\tilde{I}, y) = (x, y, z^2, zw^k, w^{k+1})$. Similar to the previous case, $ \lambda_R(R/(x, y, z^2,  zw^k, w^{k+1}))$ is 
$2k+1$ and $ \lambda_R(R/(x,y,z,w^{k+1})) = k + 1$. Hence we obtain
$$  \lambda_R(R/(\tilde{I}:y)) +  \lambda_R(R/(\tilde{I}, y)) \leq (k+1)+ (2k+1) = a.$$
To show that $P$ is minimally generated by the listed elements in each case, we can compute $\mu(P)= \lambda_R(P/\m P)$. In fact, if we let $\bar{R}=R/xR$, then
$$
\mu(P\bar{R})= \lambda_R(P\bar{R}/\m P\bar{R}) =  \lambda_R(P +(x)/\m P + (x)) 
= \lambda_R(P/\m P + P\cap(x)).
$$
But $P\cap(x) = x(P:x)=xP \subset \m P$. Thus $\mu(P\bar{R})= \mu(P)$. Therefore, to compute the minimal number of generators in each case, we can go modulo $(x)$ first.  If $a=3k$, we will show in Theorem \ref{main} that $P^{2} \not = P^{(2)}$, hence $P$  is not a complete intersection ideal. Thus it cannot have fewer number of generators than $4$. If $a=3k+2$, we will show  in Theorem \ref{main} that  $P$ is generated by the $4$ by $4$ Pfaffians of a $5$ by $5$ skew-symmetric matrix. Hence by Buchsbaum-Eisenbud structure theorem for height $3$ Gorenstein ideals in \cite{EB}, $P$ is minimally generated by the listed elements in this case. Thus, we only need to deal with the case $a=3k+1$, where one can check directly that the ideal $P\bar{R}$ is minimally generated by
$
z^2-yw,\; yz,\; y^2,\; w^{k+1},\; zw^k,\; yw^k.
$
\end{proof}

\begin{Theorem}\label{main}
Let $A$ be the monomial curve $A=\textsf{k}[[t^a, t^{a+r}, t^{a+2r}, t^{a+3r}]]$, where $a$ and $r$ are positive integers that are relatively prime. Regard $A$ as $R/P$, where  $R=\textsf{k}[[x,y,z,w]]$ and $P$ is the defining ideal of $A$.  
\begin{enumerate}
\item If $a=3k$ or $3k+1$, then $R/P$ is not Gorenstein and $P^{2} \not = P^{(2)}$. 
\item If $a=3k+2$, then $R/P$ is Gorenstein, $P^{2} = P^{(2)}$ and $P^{3} \not = P^{(3)}$.
\end{enumerate}
\end{Theorem}

\begin{proof} 
If $a=3k$, then one can see that $P$ contains the 2 by 2 minors of 
$$
M =
\left[
\begin{array}{ccc}
x & y & z\\
y & z & w\\
zw^{k-1} & x^{k+r} & yx^{k+r-1}\\
\end{array}
\right].
$$
Let $D = \det (M)$.  Note that $ D \not  \in P^2$, since $D$ is not in $P^2$  modulo $(x, y)$. We will show that $D \in P^{(2)}$.
We have $\det (\adj (M)) \cdot D = D^3$, where $\adj (M)$ is the adjoint matrix of $M$. Note that $D \not = 0$, for example it is not zero modulo $(x,y)$. Thus $D^2=\det (\adj (M))$.  But $\det (\adj (M)) \in P^3$, since the entries of $\adj (M)$ are in $P$.
Hence $D^2 \in P^3$. Therefore, the image of $D^2$ in the associated graded ring $G_P:={\rm gr}_{PR_P}(R_P)$ is zero. Note that $G_P$ is a domain as $R_P$ is a regular local ring. Hence the image of $D$ is zero in $G_P$, which shows that the image of $D$ in the localization $R_P$ is in $P^2R_P$, i.e., $D \in P^{(2)}$.  One could also directly show that $w\cdot \det (M) \in P^2$, hence $\det (M) \in P^{(2)}$, as $w$ is not in $P$. Now by Herzog's theorem \cite[Satz 2.8]{He}, we conclude that $R/P$ is not Gorenstein. We note that since in Proposition \ref{gens} we have shown that $P$ is minimally generated by $4$  elements, we could also use Buchsbaum-Eisenbud structure theorem for height $3$ Gorenstein ideals in \cite{EB}, or Bresinsky's  result in \cite{Bre}  which states that if a monomial curve in dimension $4$ is Gorenstein, then  $P$ is minimally generated by $3$ or $5$ elements. \\

\noindent
If $a=3k+1$, then  $P$ contains the $2$ by $2$ minors of 
$$
M =
\left[
\begin{array}{ccc}
x & y & z\\
y & z & w\\
zw^{k-1} & w^{k} & x^{k+r}\\
\end{array}
\right].
$$
With a similar argument as in the previous case, one can show that ${\rm det}(M) \in P^{(2)} \setminus P^2$. Thus by Herzog's result, $R/P$ is not Gorenstein.\\

\noindent
If $a=3k+2$, then by Proposition \ref{gens}, one can see that $P$ is generated by the $4$ by $4$ Pfaffians of 
$$
M =
\left[
\begin{array}{ccccc}
0   & -w^k     & 0       & x & y\\
w^k & 0        & x^{k+r} & y & z\\
0   & -x^{k+r} & 0       & z & w\\
-x  & -y       & -z      & 0 & 0\\
-y  & -z       & -w      & 0 & 0\\
\end{array}
\right].
$$
\vskip .1in
\noindent Thus, by Buchsbaum-Eisenbud structure theorem for height $3$ Gorenstein ideals in \cite{EB}, we obtain that $R/P$ is Gorenstein and $P$ is minimally generated by the $5$ listed elements in Proposition \ref{gens}. Hence, $P^{2} = P^{(2)}$ by Herzog's result \cite[Satz 2.8]{He}, and $P^{3} \not = P^{(3)}$ by Huneke's result \cite[Corollary 2.6]{Hu}, as $P$ is not a complete intersection ideal.
\end{proof}

\begin{Corollary}
Question \ref{Q1} and Question \ref{con1} have  affirmative answers for monomial curves as in Theorem \ref{main}.
\end{Corollary}

Now we consider monomial curves in higher dimensions.

\begin{Theorem} \label{higher}
Let $A$ be the monomial curve $\textsf{k}[[t^{a_1}, \ldots, t^{a_d}]]$.  Consider $A$ as $R/P$, where  $R=\textsf{k}[[x_1,\ldots, x_d]]$ and $P$ is the defining ideal of $A$. If  $\{a_i\}$ has an arithmetic  subsequence of length $5$, whose terms are not necessarily consecutive, then $ P^2 \not =P^{(2)}$. 
\end{Theorem}

\begin{proof}
If $\{a_i\}$ has an arithmetic subsequence $\{ b_1, \ldots, b_5\}$ of length $5$, without loss of generality we may assume that $x_1, \ldots, x_5$ correspond to $t^{b_1}, \ldots, t^{b_5}$.
Then, one can see that $P$ contains the $2$ by $2$ minors of 
$$
M = 
\left[
\begin{array}{ccc}
x_1 & x_2 & x_3\\
x_2 & x_3 & x_4\\
x_3 & x_4 & x_5\\
\end{array}
\right].
$$
We observe that $\det (M) \not \in P^2$, since $\det (M)$ is a  polynomial of degree $3$ and the generators of $P^2$ have degree at least $4$ as $P \subset \m^2$. Also note that $\det (M) \not = 0$, for example it is not zero modulo $(x_2, x_3)$.
Thus, by a similar argument as in the proof of Theorem \ref{main}, one can show  that $\det (M)  \in P^{(2)}$.
\end{proof} 

\begin{Corollary}
Question \ref{Q1} and Question \ref{con1} have positive answers for monomial curves as in Theorem \ref{higher}.
\end{Corollary}


Using a result of Morales \cite[Lemma 3.2]{Mo}, we can extend  Theorems \ref{main} and \ref{higher} to a larger class of monomial curves. As before, let $A$ be the monomial curve $\textsf{k}[[t^{a_1}, \ldots, t^{a_d}]]$. In the following we will not assume any particular order on  the $a_i$'s.
Write $A$ as $R/P$, where  $R$ is $\textsf{k}[[x_1,\ldots, x_d]]$ and $P$ is the defining ideal of $A$. For a positive integer $c$, relatively prime to $a_1$,  let $\tilde{A}$ be the modified monomial curve $\textsf{k}[[t^{a_1}, t^{ca_2}, \ldots, t^{ca_d}]]$. 
Note that $a_1, ca_2, \ldots, ca_d$ non-redundantly generate their numerical semigroup too.
Write $\tilde{A}$ as $\tilde{R}/\tilde{P}$, where  $\tilde{R}$ denotes $\textsf{k}[[x_1,\ldots, x_d]]$ and $\tilde{P}$ is the defining ideal of $\tilde{A}$.
Consider $\tilde{R}$ as an $R$-module via the map $\phi: R \longrightarrow \tilde{R}$ that sends
$x_1$ to $x_1^c$  and fixes $ x_i$  for all $i \not = 1$. For a polynomial $f(x_1, \ldots, x_d) \in R$, let $\tilde{f}$ be the polynomial  $f(x_1^c, \ldots, x_d)$.

\begin{Lemma}\label{mor} \hskip -.05in {\rm(Morales)}
$\tilde{R}$ is a faithfully flat extension of $R$. Moreover,
$P\tilde{R}\cap R = P$ and $\tilde{P}=P\tilde{R}$. In fact, $f\in P$ if and only if $\tilde{f} \in \tilde{P}$, and if $\{g_i\}$ is a minimal generating set for $P$, then $\{\tilde{g_i}\}$ is a minimal generating set for $\tilde{P}$. In addition, for all positive integers $k$, $f\in P^k$ if and only if $\tilde{f} \in \tilde{P}^k$, and $f\in P^{(k)}$ if and only if $\tilde{f} \in \tilde{P}^{(k)}$,
i.e., $\tilde{P}^k \cap R = P^k$ and $\tilde{P}^{(k)} \cap R = P^{(k)}$.
\end{Lemma}

Using Lemma \ref{mor}, we obtain the following extension of Theorems \ref{main} and \ref{higher}. 

\begin{Corollary}\label{deform}
If  Question \ref{Q1} has an  affirmative answer for a monomial curve $A$, then it also has an  affirmative answer for the monomial curve $\tilde{A}$. In particular,  Question \ref{Q1} has an  affirmative answer for successive modifications of the monomial curves as in Theorems \ref{main} and \ref{higher} in the sense of Morales.
\end{Corollary}
\begin{proof}
If $\tilde{P}^{n} = \tilde{P}^{(n)}$ for all positive integers $ n \leq d-1$, then by Lemma \ref{mor}, we obtain that $P^{n} = P^{(n)}$ for all $ n \leq d-1$. Thus, by hypothesis, $P$ is a complete intersection and hence, by Lemma \ref{mor}, we obtain that $\tilde{P}$ is a complete intersection.
\end{proof}

\vskip .3in
\section{Low Multiplicities}\label{low}

Let $R$ be a regular local ring with  maximal ideal $\m$ and of dimension $d$. Let $P$ be a prime ideal of height $d-1$. We will show that  Question \ref{Q1} has an affirmative answer when the Hilbert-Samuel multiplicity $e(R/P)$ is sufficiently small. 

\begin{Theorem}\label{lowd}
Let $R$ be a  regular local  ring with  maximal ideal $\m$  and of dimension $d$. Assume $P$ is a  prime  ideal  of height $d-1$ such that $P \subset \m^2$. 
Then  $ P^{n} \not =P^{(n)}$ for a positive integer $n$, if
$$e(R/P)  <    \prod_{r=0}^{d-2}\,  \frac{2n+r}{ n+r}.$$
\end{Theorem}
\begin{proof}
We may assume the residue field of $R$ is infinite, see for instance \cite[Lemma 8.4.2]{HS}. 
Thus, as $R/P$ has dimension one, there exists $x \in R$ whose image in $R/P$ is a minimal reduction  of $\m/P$. Note that $x$ cannot be in $\m^2$ by Nakayama's Lemma, hence $R/(x)$ is regular. Recall that in a regular local ring $S$ with maximal ideal $\mathfrak{n}$ and  of dimension $k$, $ \lambda_S(S/\mathfrak{n}^n) =  {n+k-1 \choose k}$ for all positive integers $n$. Therefore, since  $P^{n} \subset \m^{2n}$, we have
$
 \lambda_R(R/(P^{n}, x)) \geq  \lambda_R(R/(\m^{2n}, x)) =  {2n+d-2 \choose d-1}.
$
On the other hand, since $R/P$ is a one-dimensional Cohen-Macaulay ring, using the associativity formula for multiplicities, we obtain
$$ 
\begin{array}{rcl}
\lambda_R(R/(P^{(n)}, x)) &=& e((x), R/P^{(n)})\\
& =&  \lambda_{R_P}(R_P/P^{n} R_P)\cdot e((x), R/P)\\
&=&{n+d-2 \choose d-1} \cdot e(R/P).
\end{array}
$$
The multiplicity bound in the statement is equivalent to $  {n+d-2 \choose d-1} \cdot  e(R/P) <{2n+d-2 \choose d-1}$. 
Therefore, $ \lambda_R(R/(P^{(n)}, x)) <  \lambda_R(R/(P^{n}, x))$. Thus, $P^{n}$ and $P^{(n)}$  cannot be the same.
\end{proof}
 
One can easily observe that the multiplicity bound in Theorem \ref{lowd} is increasing with respect to $n$. Thus, letting $n=d-1$, we obtain the largest bound that guarantees $P^{d-1} \not = P^{(d-1)}$. Therefore, we have the following corollary.

\begin{Corollary}\label{qlow}
Under the assumptions of  Theorem \ref{lowd}, Question \ref{Q1} has a positive answer provided
$$e(R/P)  <  \prod_{r=0}^{d-2}\, \frac{2d+r-2}{ d+r-1}.$$
\end{Corollary}
Note that the multiplicity bound in Corollary \ref{qlow} grows at least exponentially with respect to $d$, since each term of the product is greater than $\frac{3}{2}$.
The next corollary is an application of Theorem \ref{lowd}   in the case of monomial curves in embedding dimension $4$.

\begin{Corollary}\label{low}
Let $A=\textsf{k}[[t^{a_1}, t^{a_2}, t^{a_3}, t^{a_4}]]$. Consider $A$ as $R/P$, where  $R=\textsf{k}[[x,y,z,w]]$ and $P$ is the defining ideal of $A$. If $a_1 = 4$ or  $5$, then
$ P^3\not = P^{(3)}$. Therefore, Question \ref{Q1} has a positive answer in this case.
\end{Corollary}
\begin{proof} Apply Theorem \ref{lowd} for $n=3$ and $d=4$. On the one hand $e(R/P)= a_1 \leq 5$, and on the other hand the multiplicity bound reduces to $ 5.6$.
Hence  $ P^3\not = P^{(3)}$. 
\end{proof}
We remark  that, by Corollary \ref{deform}, Question \ref{Q1} has an  affirmative answer for successive modifications of the monomial curves as in Corollary \ref{low} in the sense of Morales.

\vskip .3in
\section{Remarks}

We end this paper with some remarks and observations on equality of the ordinary and symbolic powers of ideals.

\begin{remark}
The  multiplicity bound in Theorem \ref{lowd} approaches $2^{d-1}$ as $n$ tends to infinity. Thus, if $e(R/P)  < 2^{d-1}$, then $P^{n} \not =P^{(n)} $ for $n$ large. 
Hence, if Question \ref{Q1} has a positive answer and $ P^{n}=P^{(n)} $ for all  $n \leq d-1$, then $e(R/P)  \geq 2^{d-1}$. 
This is consistent with the conclusion of Question \ref{Q1}, that $P$ is a complete intersection. 
To see this, suppose $P$ is generated by a regular sequence $a_1,\ldots, a_{d-1}$ and $x$ is a minimal reduction of $\m/P$ in $R/P$. Then, by  \cite[Theorem 14.9]{Mats}, we have
$$
e(\m,R/P) =  \lambda_R (R/(P,x)) =  \lambda_R (R/(a_1,\ldots, a_d)) \geq \prod_{i=1}^d {\rm ord}_{\m}(a_i) \geq 2^{d-1},
$$
where $a_d = x$. Note that ${\rm ord}_{ \m} (x) = 1$ and ${\rm ord}_{\m}(a_i) \geq 2$ for $i=1, \ldots, d-1$, as we are assuming $P \subset \m^2$.
\end{remark}

\begin{remark}
We know  that if $P^{n}  = P^{(n)}$ for $n$ large, then $P$ is a complete intersection \cite{CN}. The conclusion is also true if $P^{n}  = P^{(n)}$  for infinitely many $n$, see for instance Brodmann's result on stability of associated primes of $R/P^n$ in  \cite{Brod}. This can also be obtained using superficial elements, at least when $R$ has infinite residue field and $P$ has positive grade. If $P^{n}  = P^{(n)}$  for infinitely many $n$, then one can show that $P^{n}  = P^{(n)}$ for $n$ large. To see this, let $ x \in P$ be a superficial element, in the sense that $P^{n+1}: x = P^n$ for $n$ large, see \cite[Proposition 8.5.7]{HS}. Hence, if there exists an element $b \in P^{(n)} \setminus P^{n}$, then we have $xb \in P^{(n+1)} \setminus P^{n+1}$ for $n$ large.
\end{remark}

\begin{remark}
If $P^{n}  = P^{(n)}$ for $n$ large, then the analytic spread of $P$
is at most $d-1$  \cite{Brod2}. We note that this can also be seen via $\varepsilon$-multiplicity for one-dimensional primes. For a prime ideal $P$ of height $d-1$, we have  
$$ {\rm H}^0_{\m}(R/P^n) = P^{(n)}/P^n, $$ 
where the left hand side is the zero-th local cohomology of $R/P^n$ with support  in $\m$. Thus, if $P^{n}  = P^{(n)}$ for $n$ large, then $\varepsilon$-multiplicity of $P$ is zero, where
$$
\varepsilon(P)= \limsup_n \, \frac{d!}{n^d} \cdot  \lambda_R({\rm H}^0_{\m}(R/P^n)).
$$
Hence, by \cite[Theorem 4.7]{KV} or \cite[Theorem 4.2]{UV1}, the analytic spread of $P$ is at most~$d-1$. 
\end{remark}

\vskip .3in
\section*{Acknowledgements}  
This research was initiated at the MRC workshop on commutative algebra at Snowbird, Utah, in the summer of 2010. We would like to thank AMS, NSF and MRC for providing funding and a stimulating research environment. We are also thankful to the organizers of the workshop, David Eisenbud, Craig Huneke, Mircea Mustata and Claudia Polini for many helpful ideas and suggestions regarding this work. 

\vskip .3in

\renewcommand{\baselinestretch}{1}
\renewcommand{\refname}{References}
\renewcommand{\thepage}{\arabic{page}}

\end{document}